\documentclass{article}

\usepackage[english]{babel}
\usepackage{wrapfig}

\usepackage[letterpaper,top=2cm,bottom=2cm,left=3cm,right=3cm,marginparwidth=1.75cm]{geometry}

\usepackage{amsmath}
\usepackage{graphicx}
\usepackage[colorlinks=true, allcolors=blue]{hyperref}
\usepackage[citestyle=alphabetic,bibstyle=authortitle]{biblatex}

\usepackage{mathrsfs}
\usepackage{euscript}
\usepackage{amsfonts,latexsym,amsthm,amssymb,amsmath,amscd,euscript}
\usepackage{fancyhdr}

\usepackage{xcolor}

\usepackage{listings}

\definecolor{codegreen}{rgb}{0,0.6,0}
\definecolor{codegray}{rgb}{0.5,0.5,0.5}
\definecolor{codepurple}{rgb}{0.58,0,0.82}
\definecolor{backcolour}{rgb}{0.95,0.95,0.92}
\DeclareSymbolFont{bbold}{U}{bbold}{m}{n}
\DeclareSymbolFontAlphabet{\mathbbold}{bbold}

\lstdefinestyle{mystyle}{
  backgroundcolor=\color{backcolour}, commentstyle=\color{codegreen},
  keywordstyle=\color{magenta},
  numberstyle=\tiny\color{codegray},
  stringstyle=\color{codepurple},
  basicstyle=\ttfamily\footnotesize,
  breakatwhitespace=false,         
  breaklines=true,                 
  captionpos=b,                    
  keepspaces=true,                 
  numbers=left,                    
  numbersep=5pt,                  
  showspaces=false,                
  showstringspaces=false,
  showtabs=false,                  
  tabsize=2
}

\title{Asymptotic Distribution of Residues in Pascal's Triangle mod $p$}
\author{Connor Lane}

\newtheorem{theorem}{Theorem}[section]
\newtheorem*{proposition}{Proposition}
\newtheorem{conjecture}[theorem]{Conjecture}
\newtheorem{lemma}[theorem]{Lemma}
\newtheorem{defn}[theorem]{Definition}
\newtheorem{corollary}[theorem]{Corollary}

\theoremstyle{remark}
\newtheorem{remark}{Remark}[section]

\newcommand{\pc}{\phi_\chi}
\newcommand{\rc}{\rho_\chi}
\newcommand{\tc}{\theta_\chi}

\begin{document}
\maketitle

\begin{abstract}
    Fix a prime $p$ and define $T_p(n)$ to be the number of nonzero residues in the $n$th row of pascal's triangle mod $p$, and define $\phi_p(n)$ to be the number of nonzero residues in the first $n$ rows of pascal's triangle mod $p$. We generalize these to sequences $T_\chi(n)$ and $\phi_\chi(n)$ for a Dirichlet character $\chi$ of modulus $p$. We prove many properties of these sequences that generalize those of $T_p(n)$ and $\phi_p(n)$. Define $A_n(r)$ to be the number of occurrences of $r$ in the first $n$ rows of Pascal's triangle mod $p$. Guy Barat and Peter Grabner showed that for all primes $p$ and nonzero residues $r$, $A_n(r)\sim \frac{1}{p-1}\phi_p(n)$. We provide an alternative proof of this fact that yields explicit bounds on the error term. We also discuss the distribution of $A_p(r)$.
\end{abstract}

\section{Introduction}

The problem of the structure of Pascal's triangle mod $p$ has a long history, starting with the following theorem of Lucas \cite{10.2307/2304500}. Suppose $n$ has $p$-ary expansion $\overline{n_kn_{k-1}\hdots n_0}$ and $m$ has $p$-ary expansion $\overline{m_km_{k-1}\hdots m_0}$. Then

\begin{equation*}
    \binom{n}{m}\equiv \prod_{j=0}^k \binom{n_j}{m_j} \mod{p}.
\end{equation*}

This reduces computation of $\binom{n}{m}\mod{p}$ to computing $\binom{n_j}{m_j} \mod{p}$, where $n_j,m_j< p$. Motivated by this, we define

\begin{defn}
    The fundamental domain of pascal's triangle mod $p$ is the first $p$ rows of the triangle.
\end{defn}

Next, we introduce some notation:

\begin{enumerate}
    \item $p$ is a fixed prime unless otherwise specified. Function definitions are always defined in terms of the choice of $p$, even if not explicitly specified.
    \item $T_p(n)$ is the number of nonzero residues in the $n$th row of Pascal's triangle mod $p$.
    \item $\phi_p(n)$ is the number of nonzero residues in the first $n$ rows of Pascal's triangle mod $p$. 
    \item $a_n(r)$ is the number of occurrences of $r$ in the $n$th row of Pascal's triangle mod $p$, where the triangle is understood to start at the zeroeth row.
    \item $A_n(r)=\sum_{u=0}^{n-1}a_n(r)$ is the number of occurences of $r$ in the first $n$ rows of Pascal's triangle mod $p$.
    \item $\chi$ is always a Dirichlet character with modulus $p$.
\end{enumerate}

In 2001, \cite{2001paper} proved the following theorem

\begin{theorem}\label{the:1.2}
    Suppose $p$ is a prime and $r$ is a nonzero residue mod $p$. Then as $n$ goes to infinity,
    \begin{equation*}
        A_n(r)\sim \frac{\phi_p(n)}{p-1}.
    \end{equation*}
\end{theorem}

In fact, they proved a generalization to prime powers and the $p$th-power free part of binomial coefficients. However, we focus on this special case in our paper, and using alternative methods we prove the following asymptotic bounds on $A_n(r)$.

\begin{theorem}\label{the:1.3}
    Let $p$ be a prime and $r$ a nonzero residue mod $p$. Let $\vartheta$ be defined as in section \ref{sec:conc}. Then
    \begin{equation*}
        A_n(r)=\frac{\phi_p(n)}{p-1}+O(n^{\vartheta}).
    \end{equation*}
    Further, the constant implied by the big $O$ is explicitly computable.
\end{theorem}

In section \ref{sec:2}, we introduce two sequences determined by a Dirichlet character $\chi$, $T_{\chi}(n)$ and $\phi_\chi(n)$, which roughly correspond to $a_n(r)$ and $A_n(r)$, however, they obey some very nice identities. Then in section \ref{sec:props} we prove some asymptotic bounds on the behavior of $\phi_\chi(n)$ based on behavior in the fundamental domain. Then, in section \ref{sec:fd} we analyze the fundamental domain using a mixture of heuristic methods and concrete bounds. Finally in section \ref{sec:conc} we combine the results of section \ref{sec:props} and section \ref{sec:fd} to prove theorem \ref{the:1.3}, and we discuss some conjectures.

\section{The functions $T_\chi(n)$ and $\pc(n)$}\label{sec:2}

For a fixed prime $p$ let $a_n(r)$ be the number of occurrences of $r$ in the $n$th row of Pascal's triangle mod $p$. We define $T_\chi(n)$

\begin{equation*}
    T_\chi(n)=\sum_{j=0}^n \chi\!\left(\!\binom{n}{j}\!\right)=\sum_{i=1}^{p-1} \chi(i)a_n(i)
\end{equation*}

\begin{remark}
    In the language of \cite{2001paper}, this is a $1$-block multiplicative function. Many of the theorems we state about $T_\chi$ follow from general theorems about block-multiplicative functions, however we include their proof for completeness sake. 
\end{remark}

\begin{proposition}
    Let $n$ have $p$-ary expansion $n=\overline{n_k\hdots n_0}$, then

    \begin{equation}\label{eq:1}
        T_\chi(n)=\prod_{j=0}^k T_{\chi}(n_j)
    \end{equation}
\end{proposition}

\begin{proof}
    This is a restatement of the primary result of \cite{GARFIELD19921}. If $g$ is a generator of $(\mathbb{Z}/p\mathbb{Z})^\times$, and $a_n(r)$ is the number of occurences of $r$ in the $n$th row of Pascal's triangle mod $p$. We define the polynomial $R_n(x)$, where 
    \begin{equation*}
        R_n(x)=\sum_{i=0}^{p-2} x^i a_n(g^i).
    \end{equation*}

    then, using our notation, they showed
    \begin{equation*}
        R_n(x)\equiv \prod_{j=0}^{k} R_{n_k}(x) \mod x^{p-1}-1
    \end{equation*}

    We know $\chi(g^n)=\chi(g)^n$, so it follows that $R_n(\chi(g))=T_{\chi}(n)$, and since $\chi(g)^{p-1}-1=0$, the result follows.
\end{proof}

By considering partial sums of $T_\chi(n)$, we define $\phi_\chi(n)$:

\begin{equation*}
    \phi_\chi(n)=\sum_{u=0}^{n-1} T_\chi(u).
\end{equation*}

We remark that if $\chi_0$ is the principal character mod $p$, then $\phi_{\chi_0}(n)=\phi_p(n)$, which has been heavily studied in the literature. It equals the number of nonzero residues in the first $n$ rows of Pascal's triangle mod $p$. Among other things, it has been shown that if $\theta=\log_p(\phi_p(p))$, then $\alpha=\limsup(\tfrac{\phi_p(n)}{n^\theta})$ and $\beta=\liminf(\tfrac{\phi_p(n)}{n^\theta})$ both exist, with $\alpha=1$ and $1>\beta>0.5$ \cite{10.2307/2041936, BradWilson1998}. These theorems are made possible by certain recursive formulas for $\phi_p(n)$, the following lemma generalizes these fractal properties to arbitrary $\phi_\chi(n)$.
\newpage
\begin{lemma}\label{lem:2.1}
    \phantom \\
    \begin{enumerate}
        \item For all nonegative integers $m,k$ we have $\phi_{\chi}(mp^k)=\phi_\chi(m)\phi_\chi(p^k)$
        \item Furthermore, for all nonnegative $n<p^k$, we have $\phi_\chi(mp^k+n)=\phi_\chi(mp^k)+T_\chi(m)\phi_\chi(n)$
    \end{enumerate}
\end{lemma}

\begin{proof}
    First, we show 1. Let $m,k\in \mathbb{N}$, then we have

    \begin{align*}
        \phi_\chi(mp^k)=&\sum_{u=0}^{mp^k-1} T_\chi(u) \\
        =&\sum_{u_1=0}^{m-1} \sum_{u_2=0}^{p^k-1} T_\chi(u_1p^k+u_2) \\
        =&\sum_{u_1=0}^{m-1} \sum_{u_2=0}^{p^k-1} T_\chi(u_1)T_\chi(u_2)\\
        =&\sum_{u_1=0}^{m-1} T_\chi(u_1)\sum_{u_2=0}^{p^k-1} T_\chi(u_2)\\
        =& \phi_\chi(m)\phi_\chi(p^k)
    \end{align*}

Note that in the third line we use the fact that the last $k$ digits of $u_1p^k+u_2$ are exactly the digits of $u_2$ and all other digits are the digits of $u_1$. This completes the proof of part 1. The proof of part 2 is similar

\begin{align*}
    \phi_\chi(mp^k+n)=& \sum_{u=0}^{mp^{k}+n-1} T_\chi(u) \\
    =& \sum_{u=0}^{mp^{k}-1}T_\chi(u)+\sum_{u=0}^{n-1} T_\chi(mp^k+u) \\
    =& \phi_\chi(mp^k)+T_\chi(m)\sum_{u=0}^{n-1} T_\chi(u) \\
    =& \phi_\chi(mp^k)+T_\chi(m)\phi_\chi(n).
\end{align*}

Where we use the fact that the last $k$ digits of $mp^k+u$ are exactly the digits of $u$. This concludes the proof of part 2.
\end{proof}

Next, we let $A_n(r)$ be the number of occurrences of the residue $r$ in the first $n$ rows of pascal's triangle mod $p$. We note that $\phi_\chi(n)$ can be written in terms of $A_n(r)$:

\begin{equation*}
    \phi_\chi(n)=\sum_{u=0}^{n-1} T_\chi(n)=\sum_{u=0}^{n-1}\sum_{r=1}^{p-1} \chi(r)a_u(r)=\sum_{r=1}^{p-1} \sum_{u=0}^{n-1}\chi(r)a_u(r)=\sum_{p=1}^{n-1}\chi(r)A_u(r).
\end{equation*}

More interestingly, we can actually compute $A_n(r)$ in terms of $\phi_\chi(n)$.

\begin{lemma}\label{lem:explicitformula}
    Let $n$ be a nonnegative integer; then,
    \begin{equation*}
        A_n(r)=\frac{1}{p-1}\sum_{\chi} \overline{\chi}(r)\phi_\chi(n).
    \end{equation*}
\end{lemma}

\begin{proof}
    Let $s$ be an integer such that $sr\equiv 1 \mod p$. Then $\overline{\chi}(r)=\chi(s)$ and
    \begin{align*}
        \frac{1}{p-1}\sum_{\chi}\overline{\chi}(r)\phi_\chi(n)=&\frac{1}{p-1}\sum_\chi \sum_{t=1}^{p-1} \chi(s)\chi(t)A_n(t)=\frac{1}{p-1}\sum_{t=1}^{p-1}\sum_{\chi}\chi(st)A_n(t)\\
        =&\frac{1}{p-1}\sum_{t=1}^{p-1}\begin{cases}
            (p-1)A_n(t) & \text{if }st\equiv 1\mod p \\
            0 & \text{otherwise}
        \end{cases}\\
        =&A_n(r).
    \end{align*}

    Where we use orthogonality of characters to simplify the sum over Dirichlet characters mod $p$.
\end{proof}

This makes $A_n(r)$ significantly easier to study, as it is reduced to studying $\phi_\chi(n)$, a sequence that is much more well-behaved. 

\section{Properties of $\phi_\chi(n)$}\label{sec:props}

We begin by fixing a prime $p$ and character $\chi$ of modulus $p$. We define $\tc=\log_p(\phi_\chi(p))$, and we take the principal branch of the logarithm.

Next, we define a certain technical condition that the theorems of this section rely upon. Further discussion of this condition can be found in section \ref{sec:fd}.

\begin{defn} \label{rr}
    A character $\chi$ is called row-regular if for all $0\leq b<p$, we have $|T_\chi(b)|<|\phi_\chi(p)|$.
\end{defn}

Under a row-regularity assumption, the behavior of $\pc(n)$ is actually quite predictable, and is the focus of theorems \ref{the:3.2} and \ref{the:3.4}. Both of these follow from theorem 1 of \cite{2001paper}, which works more generally with $l$-block multiplicative functions. These proofs are nonetheless included in our paper so that we have a complete proof of theorem \ref{the:1.3}.

\begin{theorem} \label{the:3.2}
    Fix a row-regular character $\chi$. Then we have $|\phi_\chi(n)|=O(n^{\tc})$. Moreover, if we define $\alpha=\limsup(|\phi_\chi(n)/n^{\tc})|$, then $\alpha$ exists and is greater than or equal to $1$.
\end{theorem}

\begin{proof}
    We first define a sequence of positive real numbers $\{\alpha_k\}_{k>0}$ as follows
    \begin{equation*}
        \alpha_k=\max\left\{\left|\frac{\phi_\chi(n)}{n^{\tc}}\right| : p^{k-1}< n \leq p^k\right\}
    \end{equation*}
    It is clear that, if it exists, $\lim_{k\rightarrow \infty} (\alpha_k)=\alpha$. We will show that $\alpha_{k+1}\in [\alpha_k, \alpha_k+|\phi_\chi(p)|\alpha_1 q^k)$ for some $|q|<1$.  First, we will show that $\alpha_{k+1}\geq \alpha_k$. Select $p^{k-1}<n\leq p^{k}$ such that $|\phi_\chi(n)/n^{\tc}|=\alpha_k$. Then since $p^k<np\leq p^{k+1}$, we can use lemma \ref{lem:2.1} part 1 to show

    \begin{equation*}
        \alpha_{k+1}\geq \left|\frac{\phi_\chi(pn)}{(pn)^{\tc}}\right|=\left|\frac{\phi_\chi(p)\phi_\chi(p)}{p^{\tc}n^{\tc}}\right|=\left|\frac{\pc(p)\pc(n)}{\pc(p)n^{\tc}}\right|=\left|\frac{\pc(n)}{n^{\tc}}\right|=\alpha_k.
    \end{equation*}

    Next, we can show that $\alpha_{k+1}-\alpha_k\leq |\pc(p)|\alpha_1 q^k$. We consider some $n$ such that $p^k<n\leq p^{k+1}$ and $|\pc(n)/n^{\tc}|=\alpha_{k+1}$, and write $n=pm+b$ for $p^{k-1}<m\leq p^k$ and $0\leq b <p$. Then using both parts of lemma \ref{lem:2.1} we see that

    \begin{align*}
        \alpha_{k+1}=&\left|\frac{\pc(pm+b)}{n^{\tc}}\right| \\
        \leq & \frac{1}{|(mp)^{\theta_{\chi}}|}\left|\pc{p}\pc(m)+\pc(b)T_\chi(m)\right| \\
        \leq & \frac{1}{|(mp)^{\theta_{\chi}}|}\big(|\pc(p)||\pc(m)|+|\pc(b)||T_\chi(m)|\big)
    \end{align*}

    Since $b< p$, we know that $|\pc(b)|< \alpha_1 |p^{\tc}|$. We also use the fact that $|\pc(p)|=|p^{\tc}|$.
    \begin{align*}
        \alpha_{k+1}<& \frac{|\pc(m)|}{|m^{\theta_{\chi}}|}+\frac{\alpha_1 |T_\chi(m)|}{|m^{\theta_{\chi}}|} \\
        \leq & \alpha_k+\frac{\alpha_1 |T_\chi(m)|}{|(p^{k-1})^{\theta_{\chi}}|}.
    \end{align*}

    The proof of the theorem would follow if we can bound $|T_\chi(m)|/(m^{k-1})^{\tc}$. This is quite straight forward, but we will factor out to a lemma so we can reference it later.

    \begin{lemma}\label{lem:3.3}
        Let $m<p^k$ be a nonnegative integer and $\chi$ a row regular character. Then there exists some real number $0<q<1$ independent of $m$ such that
        \begin{equation*}
            \left|\frac{T_\chi(m)}{m^{\pc}}\right|\leq q^k 
        \end{equation*}
    \end{lemma}

    \begin{proof}
        Since $m$ is a $k$ digit number, we can use equation \ref{eq:1} to write $|T_\chi(m)|=\prod_{j=0}^{k-1}|T_\chi(m_j)|$, where $m_j$ are the $p$-ary digits of $m$. simply maximising each entry in the product, we have $|T_\chi(m)|\leq \prod_{j=0}^{k-1}\max\{|T_\chi(t)|:0\leq t < p\}=(\max\{|T_\chi(t)|:0\leq t < p\})^k$. With this in mind, we let $q=\max\{|T_\chi(t)|:0\leq t < p\}/\pc(p)$, and row-regularity implies $q<1$. This gives
        \begin{equation*}
            \left|\frac{T_\chi(n)}{n^{\tc}}\right|\leq \left|\frac{\max\{|T_\chi(t)|:0\leq t < p\})^k}{(p^k)^{\tc}}\right|=\left|\frac{\max\{|T_\chi(t)|:0\leq t < p\})^k}{\pc(p)^k}\right|\leq q^k.
        \end{equation*}
        Which is what we wanted to show.
    \end{proof}
    
    Using lemma \ref{lem:3.3}, we obtain

    \begin{align*}
        \alpha_{k+1}< & \alpha_k+|p^{\tc}|\alpha_1\frac{T_\chi(n)}{n^{\theta_\chi}} \\
        =& \alpha_k+|\pc(p)|\alpha_1q^k.
    \end{align*}

    This completes the proof that $\alpha_{k+1}\in [\alpha_k, \alpha_k+|\pc(p)|\alpha_1 q^k)$ for some $|q|<1$. Since the geometric series $\alpha_1+\prod_{k=1}^{\infty} |\pc(p)|\alpha_1 q^k$ converges, we have that ${\alpha_k}$ is bounded and $\lim_{k\rightarrow \infty} \alpha_k$ converges by monotone convergence theorem. This means that $\alpha=\limsup(\pc(n)/n^{\tc})$ exists. In particular, this implies that $\pc(n)=O(n^{\tc})$

    To show that $\alpha\geq 1$, we simply note that $\pc(p^k)/(p^k)^{\tc}=1$ for all $k$ by a simple application of lemma \ref{lem:2.1}.
\end{proof}

We note that the sum of the geometric series discussed at the end of that proof gives an effective upper bound for $\alpha$.

Next, we generalize a theorem of \cite{10.1007/BFb0083577} about the behavior of $\phi_{\chi_0}$ to arbitrary row-regular characters. We define the following function:

\begin{equation*}
    \psi_\chi(n)=\frac{\phi_\chi(n)}{n^\theta_\chi}.
\end{equation*}

Theorem \ref{the:3.2} implies that $\psi_\chi(n)=O(1)$, and lemma \ref{lem:2.1} implies $\psi_\chi(pn)=\psi_\chi(n)$. Using this formula, we canonically extend the domain of $\psi_\chi$ to $D=\{{n}/{p^k}:n\in \mathbb{Z}^{>0}, k\in \mathbb{Z}^{\geq 0}\}$. We remark that $D$ is dense in $\mathbb{R}^{>0}$, so if we show that $\psi_\chi$ is continuous on $D$, we get a canonical extension to $\mathbb{R}^{>0}$. Indeed, we will prove

\begin{theorem}\label{the:3.4}
    Let $\chi$ be a row-regular character. Then $\psi_\chi(x)$ is uniformly continuous on subsets bounded away from $0$.
\end{theorem}

\begin{proof}
    We will prove uniform continuity in the set $[1,\infty)$, and uniform continuity in sets bounded away from $0$ will follow by using the fact that $\psi_\chi(x)=\psi_\chi(px)$

    For nonnegative integers $a,r,k$ with $k<p^{r-a}$ and a positive integer $n$, we will bound the difference $\psi_\chi(p^rn+k)-\psi_\chi(p^rn)$ uniformly in $r$. Expanding definitions and applying \ref{lem:2.1} we obtain
    \begin{align*}
        \left|\psi_\chi(p^rn+k)-\psi_\chi(p^rn)\right|=&\left|\frac{\pc(p^rn+k)}{(p^rn+k)^{\tc}}-\frac{\pc(p^rn)}{(p^rn)^{\tc}}\right|\\
        =&\left|\frac{\pc(p^rn)}{(p^rn+k)^{\tc}}+\frac{T_\chi(n)\pc(k)}{(p^rn+k)^{\tc}}-\frac{\pc(p^rn)}{(p^rn)^{\tc}}\right|\\
        \leq &\left|\frac{\pc(p^rn)}{(p^rn+k)^{\tc}}-\frac{\pc(p^rn)}{(p^rn)^{\tc}}\right|+\left|\frac{T_\chi(n)\pc(k)}{(p^rn+k)^{\tc}}\right|
    \end{align*}

    We will bound each of these terms separately. The 2nd term is easier to bound, so we will do it first. Using \ref{the:3.2}, we obtain

    \begin{equation*}
        \left|\frac{T_\chi(n)\pc(k)}{(p^rn+k)^{\tc}}\right|\leq \left|\frac{T_\chi(n)\alpha k^{\tc}}{(p^rn)^{\tc}}\right|\leq \left|\frac{T_\chi(n)\alpha (p^{r-a})^{\tc}}{(p^rn)^{\tc}}\right|=\left|\frac{T_\chi(n)}{(p^{a}n)^{\tc}}\right|.
    \end{equation*}

    Using lemma \ref{lem:3.3}, we bound $|T_\chi(n)/n^{\tc}|\leq q^{\log_{p}(n)}=O (1/n^{\omega_\chi})$ for some $1\geq \omega_\chi>0$. (We may be able to obtain a $\omega_\chi\geq 1$ for some characters, but it will be of convenience later to restrict it to be less than $1$, and the equation is still true in that case.) This gives the bound that the second term is $O(n^{-\omega_\chi}p^{-a})$ uniformly in $r$.

    Next we bound the first term. To this end, we prove the following general lemma

    \begin{lemma}\label{lem:3.5}
        Let $a,b$ be positive real numbers with $b<a$, and $\theta\in \mathbb{C}$ have positive real part. Then
        \begin{equation*}
            \left|\frac{1}{(a+b)^\theta}-\frac{1}{a^\theta}\right|=O\left(\frac{b}{a^{1+\theta}}\right).
        \end{equation*}
    \end{lemma}

    \begin{proof}
        Some elementary algebra yields
        \begin{equation*}
            \left|\frac{1}{(a+b)^\theta}-\frac{1}{a^\theta}\right|=\left|\frac{1^\theta-\left(1+\frac{b}{a}\right)^\theta}{(a+b)^\theta}\right|\leq \left|\frac{1-(1+\frac{b}{a})^\theta}{a^\theta}\right|.
        \end{equation*}

        We then use the generalized binomial theorem to expand $(1+\frac{a}{b})^\theta$. 
        \begin{align*}
            \left|\frac{1}{(a+b)^{\theta}}-\frac{1}{a^\theta}\right|\leq & \left|\frac{1-(1+\frac{b}{a})^\theta}{a^\theta}\right| \\
            =& \left|\frac{1-\sum_{n=0}^\infty \binom{\theta}{n}\left(\frac{b}{a}\right)^n}{a^\theta}\right| \\
            =&\left|\frac{-\sum_{n=1}^\infty \binom{\theta}{n}\left(\frac{b}{a}\right)^n}{a^\theta}\right| \\
            =&O\left(\frac{b}{a^{\theta+1}}\right).
        \end{align*}

         Which is what we wanted to show.
    \end{proof}

    We now return to bounding the first term. Since $a\geq 0$, we know that $k<p^rn$, so we can apply lemma \ref{lem:3.5} to the denominator of the first term. 

    \begin{align*}
        \left|\frac{\pc(p^rn)}{(p^rn+k)^{\tc}}-\frac{\pc(p^rn)}{(p^rn)^{\tc}}\right|=&|\pc(p^rn)|\left|\frac{1}{(p^rn+k)^{\tc}}-\frac{1}{(p^rn)^{\tc}}\right|\\
        \leq& |(p^rn)^{\tc}|\left|\frac{1}{(p^rn+k)^{\tc}}-\frac{1}{(p^rn)^{\tc}}\right|\\
        =& |(p^rn)^{\tc}|O\left(\frac{k}{(p^rn)^{\tc+1}}\right) \\
        =& O\left(\frac{p^{r-a}}{(p^rn)}\right)=O\left(\frac{1}{np^a}\right).
    \end{align*}

    Where for the last inequality, we use the fact that $n$ is a positive integer to bound it by $1$.

    Combining this with the bound on the first term, we obtain a bound that goes to $0$ as $a$ goes to infinity

    \begin{equation*}
        \left|\psi_\chi(p^rn+k)-\psi_\chi(p^rn)\right|=O\left(\frac{1}{n^{\omega_\chi}p^a}\right)+O\left(\frac{1}{np^a}\right)=O\left(\frac{1}{n^{\omega_\chi}p^a}\right).
    \end{equation*}

    Let $x_0=\frac{n}{p^b}\geq 1$. If $0\leq x-x_0< p^{1-b}$, we will bound $\psi_\chi(x)-\psi_\chi(x_0)$ in such a way that it goes to $0$ as $b$ goes to infinity. This will be the last ingredient needed for uniform continuity.

    Since $x_0\geq 1$, we obtain $n\geq p^b$. Further, we write $x-x_0=k/(p^{b-1+r})$ with $k<p^{r}$. using the fact that $\psi_\chi(x)=\psi_\chi(px)$, we see that
    \begin{equation*}
        |\psi_\chi(x)-\psi_\chi(x_0)|=\left|\psi_\chi\left(n+\frac{k}{p^{1+r}}\right)-\psi_\chi(n)\right|=|\psi_\chi(p^{1+r}n+k)-\psi_\chi(p^{1+r}n)|=O\left(\frac{1}{n^{\omega_\chi}p}\right)=O\left(\frac{1}{p^{b\omega_\chi}}\right).
    \end{equation*}

    Finally, we move towards uniform continuity. Let $x_0\geq 1$. And let $|x-x_0|< \frac{1}{p^b}$. Define $y_0=\frac{k}{p^b}$ to be an element of $\frac{1}{p^b}\mathbb{Z}$ such that $y_0<x_0,x$ and $|y_0-x|<\frac{1}{p^{b-1}}$ and $|y_0-x_0|<\frac{1}{p^{b-1}}$. Then the bounds above imply that

    \begin{equation*}
        |\psi_\chi(x)-\psi_\chi(x_0)|=|\psi_\chi(y_0)-\psi_\chi(x)|+|\psi_\chi(y_0)-\psi_\chi(x_0)|=2O\left(\frac{1}{p^{b\omega_\chi}}\right)
    \end{equation*}

    Since the right side goes to $0$ uniformly in $x_0$ as $a\rightarrow \infty$, it follows that $\psi_\chi(x)$ is uniformly continuous in $[1,\infty)$. Uniform continuity on sets bounded away from $0$ follows as mentioned in the beginning of the proof.
\end{proof}
        
We also have a near-inverse of theorem \ref{the:3.2}, that only leaves out a tiny edge case. To do this, we introduce a new definition.

\begin{defn}
    A \textbf{row-dominant} character $\chi$ is a character $\chi$ of modulus $p$ such that there is a $b<p$ such that $|T_\chi(n)|>|\phi_\chi(p)|$.
\end{defn}

\begin{remark}
    A character $\chi$ would be neither row-regular or row-dominant if there is a $0\leq b<p$ such that $|T_\chi(b)|=|\pc(p)|$, but there is no $0\leq b<p$ such that $|T_\chi(b)|>|\pc(p)|$.
\end{remark}

\begin{theorem} \label{the:3.7}
    If $\chi$ is a row-dominant character, then $\pc(n)$ is not $O(n^{\tc})$.
\end{theorem}

\begin{proof}
    Suppose for the sake of contradiction that $\pc(n)$ is $O(n^{\tc})$. Then let $b$ be the integer such that $0\leq b<p$ and $|T_\chi(b)|>|\pc(p)|$ whose existence is guaranteed by row-dominance. Now define the integer sequence $\{n_k\}_{k>0}=(\sum_{j=0}^{k-1} bp^j)$. We then compute $\pc(n_k+1)-\pc(n_k)$

    \begin{equation*}
        \pc(n_k+1)-\pc(n_k)=T_\chi(n_k)=\prod_{i=0}^{k-1} T_{\chi}(b)=T_\chi(b)^k
    \end{equation*}

    Where in the last equality we use equation \ref{eq:1}. Therefore

    \begin{align*}
        \frac{\pc(n_k+1)}{(p^k)^{\tc}}-\frac{\pc(n_k)}{(p^k)^{\tc}}=&\frac{T_\chi(b)^k}{(p^k)^{\tc}} \\
        \left|\frac{\pc(n_k+1)}{(p^k)^{\tc}}\right|+\left|\frac{\pc(n_k)}{(p^k)^{\tc}}\right|\geq &\left|\frac{T_\chi(b)^k}{(p^k)^{\tc}} \right|.
    \end{align*}

    Now, by our assumption, for sufficiently large $n$ we have $|\pc(n)/(n^{\tc})|\leq \alpha$ for some real number $\alpha$. We also note that $|\pc(n_k+1)/(p^{k})^{\tc}|\leq |(\pc(n_k+1)/(n_k+1)^{\tc}|\leq \alpha$ for sufficiently large $k$. This means we have

    \begin{align*}
        2\alpha \geq \left|\frac{T_\chi(b)^k}{(p^{k})^{\tc}}\right| \geq \Bigr|\frac{T_\chi(b)^k}{\pc(p)^k}\Bigr|.
    \end{align*}

    However, since $|T_\chi(b)|>|\pc(p)|$, the right hand side of the above equation is unbounded, so it cannot be bounded by $2\alpha$. This is a contradiction, so $\pc(n)$ is not $O(n^{\theta_{\chi}})$.
    
\end{proof}

There is one more theorem on the growth rate of $\pc(n)$. This one allows us to bound the growth rate of $\pc(n)$ for non-row-regular characters $\chi$. This is where our method to prove theorem \ref{the:1.2} deviates from the one presented in \cite{2001paper}. In their proof, they avoided the non-row-regular case by working with bivariate block multiplicative functions. We instead handle the non-row-regular case directly. To do this, we will define the real number $\rho_\chi$

\begin{equation*}
    \rc=\max\{\mathfrak{R}(\log_p(T_\chi(b))):0\leq b<p\}.
\end{equation*}

\newpage

\begin{theorem} \label{the:3.8}
    Let $\chi$ be not row-regular and let $\varepsilon>0$. then $\pc(n)=O(n^{\rc+\varepsilon})$.
\end{theorem}

\begin{proof}
    This proof follows an outline very similar to \ref{the:3.2}. We define a sequence $\{\alpha_k\}_{k>0}$ as
    \begin{equation*}
        \alpha_k=\max\left\{\left|\frac{\pc(n)}{n^{\rc+\varepsilon}}\right| : p^{k-1}< n \leq p^k\right\}
    \end{equation*}

    We will show that $\alpha_{k+1}\leq \alpha_k+p^{\rc+\varepsilon}\alpha_1q^k$ for $q<1$. We consider some $n$ such that $p^k<n\leq p^{k+1}$ and $|\pc(n)/n^{\rc+\varepsilon}|=\alpha_{k+1}$, and write $n=pm+b$ for $p^{k-1}<m\leq p^k$ and $0\leq b <p$. Then using both parts of lemma \ref{lem:2.1} we see that

    \begin{align*}
        \alpha_{k+1}=&\left|\frac{\pc(pm+b)}{n^{\rc+\varepsilon}}\right| \\
        \leq & \frac{1}{(pm)^{\rc+\varepsilon}}\left|\pc(p)\pc(m)+\pc(b)T_\chi(m)\right| \\
        \leq & \frac{1}{(pm)^{\rc+\varepsilon}}\big(|\pc(p)||\pc(m)|+|\pc(b)||T_\chi(m)|\big).
    \end{align*}

    Since $b\leq p$ we know $\frac{\pc(b)}{p^{\rc+\varepsilon}}\leq \alpha_1$. Furthermore, since $\chi$ is not row regular, $\frac{\pc(p)}{p^{\rc+\varepsilon}}\leq 1$. Therefore

    \begin{align*}
        \alpha_{k+1}\leq&\frac{|\pc(m)|}{m^{\rc+\varepsilon}}+\alpha_1\frac{|T_\chi(m)|}{m^{\rc+\varepsilon}} \\
        \leq& \alpha_k+\alpha_1\frac{T_\chi(m)}{(p^{k-1})^{\rc+\varepsilon}}
    \end{align*}

    Since $m$ is a $k$ digit number, we can use \ref{eq:1} to write $|T_\chi(m)|=\prod_{j=0}^{k-1}|T_\chi(d_j)|$, where $d_j$ are the $p$-ary digits of $m$. Taking the largest possible value of $|T_\chi(m)|$, we maximise each entry in the product to get $|T_\chi(m)|\leq \prod_{j=0}^{k-1}\max\{|T_\chi(t)|:0\leq t < p\}=(\max\{|T_\chi(t)|:0\leq t < p\})^k$. With this in mind, we let $q=\max\{|T_\chi(t)|:0\leq t < p\}/(p^{\rc+\varepsilon})$, and the definition of $\rc$ implies $q<1$. This gives

    \begin{align*}
        \alpha_{k+1}\leq & \alpha_k+|p^{\rc+\varepsilon}|\alpha_1\frac{(\max\{|T_\chi(t)|:0\leq t < p\})^k}{(p^{k})^{\rc+\varepsilon}} \\
        \leq & \alpha_k+p^{\rc+\varepsilon}\alpha_1q^k.
    \end{align*}

    Therefore $\alpha_{k+1}\leq \alpha_k+p^{\rc+\varepsilon}\alpha_1q^k$. Since the geometric series $\sum_{k=1}^\infty p^{\rc+\varepsilon}\alpha_1q^k$ converges, $\{\alpha_k\}$ must have an upper bound, which means that $\pc(n)=O(n^{\rc+\varepsilon})$.
\end{proof}

As with \ref{the:3.2}, the geometric series gives an effective upper bound for the constant implied by the big-O.

The behavior of $\pc(n)$ for row-dominant characters $\chi$ is extremely erratic, as some portions of it (for example the $n_k$ discussed in \ref{the:3.7}) grow faster than $O(n^{\tc})$, whereas other parts (like $p^k$) grow like $O(n^{\tc})$. However, as the previous theorem described, these are also the slowest growing $\pc(n)$, as $\rc< 1$ in general. So they do not have a significant contribution to the formula in lemma \ref{lem:explicitformula}.

\section{Row-Regularity and the Fundamental Domain of Pascal’s Triangle mod $p$}\label{sec:fd}

Thanks to Lucas' theorem and the result of \cite{GARFIELD19921}, study of Pascal's triangle mod $p$ can be reduced to understanding of it's fundamental domain, that is $\tbinom{n}{m} \text{ mod } p$ for $n,m<p$. Therefore, strong understanding of the fundamental domain leads to strong understanding of the entire triangle. This can be seen in the relative simplicity of the theory of nonzero residues in Pascal's triangle mod $p$, which largely relies on the fact that the it is easy to see if a residue in the fundamental domain is nonzero. If $n,m<p$ then $\tbinom{n}{m}\equiv 0 \text{ mod } p$ if and only if $m>n$.

With this in mind, a reasonable place to look to make progress would be by studying the fundamental domain. However, our knowledge of the fundamental domain is largely conjectural. Roughly, the fundamental domain looks like this:

\begin{center}
    \begin{tabular}{c c c c c c c}
        $1$ &  &  &  &  &  & \\
        $1$ & $1$ &  &  &  &  & \\
        $1$ & $?$ & $1$ &  &  &  & \\
        $\vdots$ & $\vdots$ & $\ddots$ & $\ddots$ & & & \\
        $1$ & $?$ & $?$ & $\ddots$ & $1$ & & \\
        $1$ & $?$ & $?$ & $\hdots$ & $?$ & $1$ & \\
        $1$ & $-1$ & $1$ & $\hdots$ & $1$ & $-1$ & $1$ \\
    \end{tabular}
\end{center}

There are $1$s running down two sides, and alternative $1$s and $-1$s on the bottom side of the triangle. Inside the triangle, there appears to be a roughly even distribution of each nonzero residue class. This suggests the following conjecture of \cite{2019Fract..2750098B}:

\begin{conjecture}\label{conj:4.1}
    As the prime modulus $p$ goes to infinity, the following asymptotics hold:
    \begin{itemize}
        \item $A_{p}(1)\sim 3p$
        \item $A_p(-1)\sim p$
        \item If $r\neq -1,0,1$, then $A_p(r)\sim\tfrac{p}{2}$
    \end{itemize}
\end{conjecture}

We wish to make a heuristic argument to motivate this conjecture and other conjectures about the fundamental domain. In particular, let $n,m<p-1$ with $m<n$, $m\neq n$. We wish to model the value of $\binom{n}{m} \text{ mod } p$ as a random variable $X_{n,m}$ taking values in $\{1,2,\hdots p-1\}$ with probability $\frac{1}{p-1}$. We assume that $X_{n_1,m_1}$ and $X_{n_2,m_2}$ are independent unless $n_1=n_2$ and $m_1=m_2$ or $m_1=n_1-m_2$, in which case they are always equal.

Using this, we can motivate conjecture \ref{conj:4.1}. 

\begin{theorem}\label{the:4.2}
    Under the assumptions of the above random model, conjecture \ref{conj:4.1} holds with probability $1$.
\end{theorem}

\begin{proof}
    Let $\mathbbold{1}_{k}(x)$ be the function that returns $1$ when $k=x$ and $0$ otherwise. Then we have
    \begin{equation*}
        A_p(r)=\sum_{n=0}^{p-1}\sum_{m=0}^{n}\mathbbold{1}_{r}\left(\binom{n}{m}\right).
    \end{equation*}
    
    The behavior for $n=p-1$, $m=n$ or $m=0$ is entirely predictable, so we consider the inside of the triangle. Since there is no inside of the triangle for $p=2,3$, we assume $p>3$ for the rest of this proof. Using our random model, we define a random variable $Y$ that determines the influence of the unpredictable inner region.
    
    \begin{equation*}
        Y=\sum_{n=2}^{p-2}\sum_{m=1}^{n-1} \mathbbold{1}_{r}(X_{n,m})=2\sum_{n=2}^{p-2}\sum_{m=1}^{\lfloor(n-1)/2\rfloor}\mathbbold{1}_{r}(X_{n,m})+\sum_{n=1}^{\lfloor(p-2)/2\rfloor} \mathbbold{1}_{r}(X_{2n,n})
    \end{equation*}
    
    The first sum is a binomial distribution with probability $\frac{1}{p-1}$ and $\frac{(p-3)^2}{4}$ trials. The second sum is a binomial distribution with probability $\frac{1}{p-1}$ and $\frac{p-3}{2}$ trials. Therefore, for any $r\neq 0$, we have
    \begin{align*}
        \mathbb{E}[Y]=&\frac{1}{p-1}\left(2\frac{(p-3)^2}{4}+\frac{p-3}{2}\right)=\frac{p^2-5p+6}{2p-2}\sim \frac{p}{2}\\
        \text{Var}[Y]=&\frac{p-2}{(p-1)^2}\left(4\frac{(p-3)^2}{4}+\frac{p-3}{2}\right)=\frac{2p^3-15p^2+37p-30}{2p^2-4p+2}\sim p
    \end{align*}
    
    Adding back in the adjustments for the outside of the triangle, we have $\mathbb{E}[A_p(1)]=\mathbb{E}[Y]+2p-1+(p+1)/2\sim 3p$, $\mathbb{E}[A_p(-1)]=\mathbb{E}[Y]+(p-1)/2\sim p$, and if $r\neq -1,0,1$, we have $\mathbb{E}[A_p(r)]=\mathbb{E}[Y]\sim p/2$. Since the standard deviation $\sigma_{A_p(r)}=\sigma_{Y}\sim \sqrt{p}=o(p)$, it follows that \ref{conj:4.1} holds with probability $1$. 
\end{proof}

Next, we turn our attention to predicting the behavior of $\pc(p)$ using the same probabilistic model. We begin essentially the same way as the previous theorem, as we have the identity
\begin{equation*}
    \pc(p)=\sum_{n=0}^{p-1}\sum_{m=0}^{n}\chi\!\left(\!\binom{n}{m}\!\right).
\end{equation*}

We once again ignore the border of the triangle as it is entirely predictable, and define a random variable $Y$ that determines the influence of the inside.
\begin{equation*}
    Y=\sum_{n=2}^{p-2}\sum_{m=1}^{n-1} \chi(X_{n,m})=2\sum_{n=2}^{p-2}\sum_{m=1}^{\lfloor(n-1)/2\rfloor} \chi(X_{n,m})+\sum_{n=1}^{\lfloor(p-2)/2\rfloor} \chi(X_{2n,n}).
\end{equation*}

We now reduce to the case where $\chi$ is nonprincipal. We see that $\chi(X_{n,m})$ is a random variable with mean $\mathbb{E}[\chi(X_{n,m})]=0$. This gives us the mean $\mathbb{E}[Y]=0$. For the variance of $Y$, We have that $\text{Var}[\chi(X_{n,m})]=1$. Further, each of the distinct $\chi(X_{n,m})$ in the sums are uncorrelated as they are independent. This gives us the variance

\begin{equation*}
    \text{Var}[Y]=4\left(\frac{(p-3)^2}{4}\right)+\frac{p-3}{2}=\frac{2p^2-11p+15}{2}\sim p^2.
\end{equation*}

Adding back the predictable component, for a nonprincipal even character $\chi$, we have that $\mathbb{E}[\pc(p)]=\mathbb{E}[Y]+3p=3p$, for an odd character, we have $\mathbb{E}[\pc(p)]=\mathbb{E}[Y]+2p+1\sim 2p$. 

This probabilistic model implies that many characters should be row regular (as a character is certainly row regular if $\pc(p)>p$.) However, the high variance implies that there should be many non-row-regular characters. These predictions turn out to fit the data quite nicely.

If we compute $\pc(p)/p$ for $\chi\neq \chi_0$ and $p<100$ (see section \ref{sec:Appendix}), and plot them on the complex plane, we see this picture:

\begin{center}
    \includegraphics[scale=0.8]{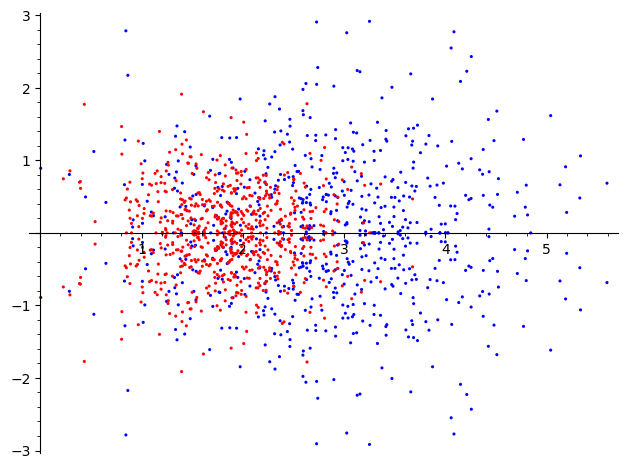}\\
    Red: $\phi_\chi(p)/p$ for $\chi(-1)=-1$. Blue: $\phi_\chi(p)/p$ for $\chi(-1)=1$.
\end{center}

And we see that most values tend to be around $2$ and $3$, though there is relatively large variance. Computation also yields the following result:

\begin{proposition}
    Not all characters $\chi$ are row-regular.
\end{proposition}

\begin{proof}
    Let $p=37$ and $\chi(2)=e^{\frac{20\pi i}{36}}$. Then a computer calculation (see section \ref{sec:Appendix}) shows
    \begin{equation*}
        \pc(p)=33e^{\frac{20\pi i}{36}}-3e^{\frac{16\pi i}{36}}-8e^{\frac{12\pi i}{36}}-21e^{\frac{8\pi i}{36}}-18e^{\frac{4 \pi i}{36}}\approx 33.7472651243456 + 2.96112697681136i.
    \end{equation*}
    Whereas
    \begin{equation*}
        T_\chi(36)=37.
    \end{equation*}
\end{proof}

\begin{remark}
    This is also the smallest non-row-regular character. It is also row-dominant.
\end{remark}

Now, we turn our attention to bounding $\pc(p)$, for nonprincipal character $\chi$, we prove an extremely weak (but nontrivial) upper bounds for $|\pc(p)|$ that is an $O(p\sqrt{p})$ improvement on the trivial bound of $p(p+1)/2$. 

\begin{theorem}\label{the:4.3}
    Fix a prime $p$ and nonprincipal character $\chi$. Then
    \begin{equation*}
        \pc(p)\leq \frac{p^2-2p\lfloor \sqrt{p} \rfloor + \sqrt{p} \lfloor \sqrt{p} \rfloor ^ 2 +p+\sqrt{p}\lfloor\sqrt{p}\rfloor + \lfloor \sqrt{p} \rfloor ^ 2 -2\sqrt{p}+\lfloor \sqrt{p} \rfloor}{2}.
    \end{equation*}
\end{theorem}

\begin{proof}
    We begin with a formula for $\pc(p)$:

    \begin{equation*}
        |\pc(p)|= \sum_{n=0}^{p-1}\sum_{m=0}^{p-1} \chi\!\left(\!\binom{m}{n}\!\right).
    \end{equation*}

    Now we apply the triangle inequality and separate this sum in to four parts, which we will analyze separately.

    \begin{equation*}
        |\pc(p)|\leq \left|\sum_{m=0}^{p-1} \chi\!\left(\!\binom{m}{0}\!\right)\right|+\left|\sum_{m=0}^{p-1} \chi\!\left(\!\binom{m}{1}\!\right)\right|+\sum_{n=2}^{\lfloor\sqrt{p}\rfloor}\left|\sum_{m=0}^{p-1} \chi\!\left(\!\binom{m}{n}\!\right)\right|+\sum_{n=\lfloor\sqrt{p}\rfloor+1}^{p-1}\left|\sum_{m=0}^{p-1} \chi\!\left(\!\binom{m}{n}\!\right)\right|.
    \end{equation*}

    For the first sum, $\binom{m}{0}=1$ for all $m$, so the first term becomes $p$. In the second term, we note that $\binom{m}{1}=m$, so the sum is $0$ by orthogonality of Dirichlet characters. In the third sum, we note that $\binom{m}{n}$ is a degree $n$ polynomial with exactly $n$ distinct roots mod $p$, so the conditions of the Weil bounds \cite{Schmidt1976} for character sums of polynomials are met. This gives us $\left|\sum_{m=0}^{p-1} \chi\left(\binom{m}{n}\right)\right|\leq n\sqrt{p}$. Finally, for the final term we know that $\binom{m}{n}=0$ for $m<n$, and since $|\chi(q)|=1$, so by the triangle inequality we have $\left|\sum_{m=0}^{p-1} \chi\left(\binom{m}{n}\right)\right|\leq \sum_{m=0}^{p-1} \left|\left(\binom{m}{n}\right)\right|\leq q-n$. Combining these together, we get

    \begin{align*}
        \pc(p)\leq& p+\sum_{n=2}^{\lfloor \sqrt{p} \rfloor} n\sqrt{p}+\sum_{n=\lfloor \sqrt{p}\rfloor + 1}^{p-1} (p-n) \\
        =& p+\sqrt{p}\frac{\lfloor \sqrt{p}\rfloor(\lfloor \sqrt{p}\rfloor + 1)}{2}-\sqrt{p}+\frac{(p-\lfloor \sqrt{p}\rfloor - 1)(p-\lfloor \sqrt{p} \rfloor)}{2} \\
        =& \frac{p^2-2\lfloor \sqrt{p} \rfloor p + \sqrt{p} \lfloor \sqrt{p} \rfloor ^ 2 +p+\sqrt{p}\lfloor\sqrt{p}\rfloor + \lfloor \sqrt{p} \rfloor ^ 2 -2\sqrt{p}+\lfloor \sqrt{p} \rfloor}{2}
    \end{align*}
\end{proof}

We can simplify the inequality to a form that looks nicer with $x - 1 < \lfloor x \rfloor\leq x$. 
\begin{corollary}
    We have the weaker but nicer looking inequality
    \begin{equation*}
        \pc(p)< \frac{p^2-p\sqrt{p}+5p-\sqrt{p}}{2}.
    \end{equation*}
\end{corollary}

\begin{proof}
    We start with \ref{the:4.3} and use $x-1 < \lfloor x \rfloor\leq x$.
    \begin{align*}
        \pc(p)\leq & \frac{p^2-2p\lfloor \sqrt{p} \rfloor  + \sqrt{p} \lfloor \sqrt{p} \rfloor ^ 2 +p+\sqrt{p}\lfloor\sqrt{p}\rfloor + \lfloor \sqrt{p} \rfloor ^ 2 +\lfloor \sqrt{p} \rfloor}{2} \\
        <& \frac{p^2-2p(\sqrt{p}-1)+\sqrt{p}^3+p+\sqrt{p}^2+\sqrt{p}^2-2\sqrt{p}+\sqrt{p}}{2} \\
        =&\frac{p^2-p\sqrt{p}+5p-\sqrt{p}}{2}.
    \end{align*}
\end{proof}

These bounds are clearly not very strong, as suggested by the diagram on page 10. However, to improve this bound we would need a much better understanding of the behavior of the fundamental domain. 

\section{Conclusion}\label{sec:conc}

We now have the necessary knowledge to prove theorem \ref{the:1.3}. We define the constant $\vartheta$ by
\begin{equation*}
    \vartheta=\max(\{\mathfrak{R}(\log_p(\pc(p):\chi\neq \chi_0))\}\cup\{1\}).
\end{equation*}

\begin{proof}[Proof of theorem \ref{the:1.3}]
   Using lemma \ref{lem:explicitformula} we write

    \begin{equation*}
        A_n(r)=\frac{1}{p-1}\sum_{\chi} \overline{\chi(r)}\pc(n).
    \end{equation*}

    Next, we bring out the $\chi=\chi_0$ term and use theorem \ref{the:3.2} on the row regular terms and theorem \ref{the:3.8} on the non-row-regular terms.

    \begin{equation*}
        A_n(r)=\frac{\phi_{\chi_0}}{p-1}+\frac{1}{p-1}\left(\sum_{\chi \text{ rr}}O(n^{\tc})+\sum_{\chi \text{ nrr}}O(n^{\rc+\varepsilon})\right).
    \end{equation*}

    Where the first sum is over row-regular pairs and the second is over non-row-regular pairs. Since $\rho_\chi<1$, we can select $\varepsilon$ such that $\rho_\chi+\varepsilon<1$ for all non-row-regular $\chi$. This gives
    
    \begin{equation*}
        A_n(r)=\frac{\phi_{\chi_0}}{p-1}+O(n^{\vartheta}).
    \end{equation*}
\end{proof}

For a fixed prime $p$, the formula $A_n(r)=\frac{\phi_p(n)}{p-1}+O(n^{\vartheta})$ is significantly better than simply theorem \ref{the:1.2}, though it requires knowledge of the fundamental domain of that prime (specifically calculating it's $\vartheta$.) Moreover, a given prime also allows us to compute constants that give explicit bounds on $A_n(r)$.  For example, if $\chi$ is the sole nonprinciple character mod $3$, then an exercise in summing the geometric series in theorem \ref{the:3.2} gives
\begin{equation*}
    |\pc(n)|\leq 6.3n^{\log_3(4)}
\end{equation*}
Using this along with \cite{BradWilson1998} yields the following bounds of $A_n(r)$ for $r\not \equiv 0 \mod 3$
\begin{align*}
    |A_n(r)-\frac{\phi_{3}(n)}{2}|\leq& 3.15n^{\log_3(4)} \\
    0.38714n^{\log_3(6)}-3.15n^{\log_3(4)}\leq A_n(r)\leq& 0.5n^{\log_3(6)}+3.15n^{\log_3(4)}.
\end{align*}

On the other hand, for an arbitrary prime, theorem \ref{the:4.3} allows us to obtain
\begin{equation*}
    \vartheta<\log_p\left(\frac{p^2-2\lfloor \sqrt{p} \rfloor p + \sqrt{p} \lfloor \sqrt{p} \rfloor ^ 2 +p+\sqrt{p}\lfloor\sqrt{p}\rfloor + \lfloor \sqrt{p} \rfloor ^ 2 -2\sqrt{p}+\lfloor \sqrt{p} \rfloor}{2}\right).
\end{equation*}

Our first conjecture is inspired by the probabilistic calculations done in section \ref{sec:fd}.

\begin{conjecture}
    Let $A_p=\{\pc(p):\chi(-1)=1,\chi\neq \chi_0\}$ and $B_p=\{\pc(p):\chi(-1)=-1\}$. Let $\mu_{A_p}$ and $\mu_{B_p}$ be the means of $A_p$ and $B_p$, then as $p$ goes to infinity,
    \begin{equation*}
        \mu_{A_p}\sim3p \qquad \mu_{B_p}\sim2p
    \end{equation*}
\end{conjecture}

\begin{remark}
    While this is inspired by the probabilistic calculations done in \ref{sec:fd}, it's not clear that this follows with probability $1$ under those assumptions. One would wish to use the central limit theorem, but that requires the assumption that $\pc(p)$ are independent for distinct $\chi$. However, this is not the case. Indeed, if $\chi$ is an injection from $\mathbb{Z}/p\mathbb{Z}$ to $\mathbb{C}$, then the value of $\pc(p)$ determines the value of $\phi_{\psi}(p)$ for any other character $\psi$.
\end{remark}

Once again inspired by the probabilistic calculations, we suspect that the value of $\pc(p)$ has quite a lot of variance, which would suggest that $|\pc(p)|<p$ quite often, which gives many opportunities for $\chi$ to be row dominant. This heuristic argument along with some numerical evidence suggest the following conjecture.

\begin{conjecture}
    There are infinitely many row-dominant characters.
\end{conjecture}

We now discuss a direction for future research, one would hope to obtain an improvement on the error term $O(n^\vartheta)$, however this cannot be directly improved, as there is a term in \ref{lem:explicitformula} that grows like $O(n^\vartheta)$, so instead we would need to bound the value of $\pc(n)$ better than just $O(n^{\tc})$. To this end, we define the zeta functions
\begin{equation*}
    Z_\chi(s)=\sum_{n=1}^\infty \frac{T_\chi(n)}{n^s}
\end{equation*}
With the hope that this function would allow us to obtain explicit formulae for $\pc(n)$.  \cite{zetafunctions} considered $Z_{\chi_0}(s)$  (among other similar functions) and obtained an explicit formula for $\psi_{\chi_0}(x)$. Theorem \ref{the:3.4} suggests we may be able to generalize their techniques to arbitrary row-regular characters $\chi$, and obtain an explicit formula for $\psi_\chi(n)$, which would yield and explicit formula for $\phi_\chi(n)$. This would give us a formula for $A_n(r)$ with an $O(n)$ error, and an exact formula for $A_n(r)$ in the case where the prime has no non-row-regular characters.

\section*{Acknowledgements}

I would like to thank Professor All for his feedback.

\printbibliography

\section{Appendix}
\label{sec:Appendix}

All code is written in python with the Sagemath library. They were run on my personal laptop (Lenovo ThinkPad P1 Gen 4). The numerical calculations should not be seen as significant numerical research, and are instead simple programs to create small visualizations and compute small counterexamples.

Code to plot $\pc(p)/p$ for $p<100$:

\begin{lstlisting}[language=Python]
def tri(p):
    A = []
    for i in range(p):
        A += [[]]
        for j in range(p):
            if i == 0 and j == 0:
                A[i] += [1]
            elif i == 0:
                A[i] += [0]
            elif j == 0:
                A[i] += [1]
            else:
                A[i] += [(A[i - 1][j - 1] + A[i - 1][j]) % p]
                
    return A

D = []
E = []
for n in range(1, 30):
    p = nth_prime(n)
    B = tri(p)
    for chi in DirichletGroup(p):
        C = []
        for i in range(p):
            C += [[]]
            for j in range(p):
                C[i] += [chi(B[i][j])]
        sum = 0
        for i in range(p):
            for j in range(p):
                sum += C[i][j]
        if (sum != p * (p + 1) / 2):
            if (chi(p - 1) == -1):
                D += [sum / p]
            else:
                E += [sum / p]
                
G = list_plot(D,size=5,color='red')
P = list_plot(E,size=5,color='blue')
P + G
\end{lstlisting}

Code to check for non-row-regular characters $\chi$.
\begin{lstlisting}[language=python]
def tri(p):
    A = []
    for i in range(p):
        A += [[]]
        for j in range(p):
            if i == 0 and j == 0:
                A[i] += [1]
            elif i == 0:
                A[i] += [0]
            elif j == 0:
                A[i] += [1]
            else:
                A[i] += [(A[i - 1][j - 1] + A[i - 1][j]) % p]
                
    return A

for n in range(1,50):
    p = nth_prime(n)
    print(p)
    B = tri(p)
    row_regular = true
    for chi in DirichletGroup(p):
        C = []
        for i in range(p):
            C += [[]]
            for j in range(p):
                C[i] += [chi(B[i][j])]
        sum = 0
        Tsum = []
        for i in range(p):
            Tsum += [0]
            for j in range(p):
                sum += C[i][j]
                Tsum[i] += C[i][j]
                
        for i in range(p):
            if abs(Tsum[i]) >= abs(sum):
                print(chi,i)
                row_regular = false
\end{lstlisting}

A search of the first 50 primes yields the following table of non-row-regular characters. We note that if $\chi$ is not row regular, $\overline{\chi}$ is also not row-regular, so we only include one character in each conjugacy pair.
    \begin{center}
        \begin{tabular}{|c|c|}
            \hline
            $p$ & $\chi$ \\
            \hline
            \hline
            $37$ & $\chi(2)=e^{\frac{20\pi i}{36}}$ \\
            \hline
            $47$ & $\chi(5)=e^{\frac{36\pi i}{46}}$ \\
            \hline
            $97$ & $\chi(5)=e^{\frac{44\pi i}{96}}$ \\
            \hline
            $97$ & $\chi(5)=e^{\frac{92\pi i}{96}}$ \\
            \hline
            $101$ & $\chi(2)=e^{\frac{56\pi i}{100}}$ \\
            \hline
            $109$ & $\chi(6)=e^{\frac{96\pi i}{108}}$ \\
            \hline
            $113$ & $\chi(3)=e^{\frac{16\pi i}{112}}$ \\
            \hline
            $131$ & $\chi(2)=e^{\frac{48 \pi i}{130}}$ \\
            \hline
            $137$ & $\chi(3)=e^{\frac{24\pi i}{136}}$ \\
            \hline
            $139$ & $\chi(2)=e^{\frac{52\pi i}{138}}$ \\
            \hline
            $139$ & $\chi(2)=e^{\frac{64\pi i}{138}}$ \\
            \hline
        \end{tabular}
        \begin{tabular}{|c|c|}
            \hline
            $p$ & $\chi$ \\
            \hline
            \hline
            $149$ & $\chi(2)=e^{\frac{52\pi i}{148}}$ \\
            \hline
            $149$ & $\chi(2)=e^{\frac{120\pi i}{148}}$ \\
            \hline
            $149$ & $\chi(2)=e^{\frac{136\pi i}{148}}$ \\
            \hline
            $151$ & $\chi(6)=e^{\frac{24\pi i}{150}}$ \\
            \hline
            $157$ & $\chi(5)=e^{\frac{60\pi i}{156}}$ \\
            \hline
            $157$ & $\chi(5)=e^{\frac{64\pi i}{156}}$ \\
            \hline
            $163$ & $\chi(2)=e^{\frac{52\pi i}{162}}$ \\
            \hline
            $173$ & $\chi(2)=e^{\frac{152\pi i}{172}}$ \\
            \hline
            $199$ & $\chi(3)=e^{\frac{116\pi i}{198}}$ \\
            \hline
            $223$ & $\chi(3)=e^{\frac{56\pi i}{222}}$ \\
            \hline
            $223$ & $\chi(3)=e^{\frac{76\pi i}{222}}$ \\
            \hline
        \end{tabular}
    \end{center}

\end{document}